\documentclass[12pt]{article}
\usepackage[english]{babel}
\usepackage[T2A]{fontenc}
\usepackage[cp1251]{inputenc}
\usepackage[tbtags]{amsmath}
\usepackage{amsmath, amsfonts,amssymb,mathrsfs,amscd}
 \usepackage{graphicx} 
 
\usepackage{amsthm}

\usepackage[mathscr]{eucal}

 \def\mf{\operatorname{mf}}

\def\simm{\operatorname{simm}}
\def\vol{\operatorname{vol}}
\def\p{\operatorname{p}}
\def\P{\operatorname{P}}
\def\Ker{\operatorname{Ker}}
\renewcommand{\Im}{\mathop{\mathrm{Im}}\nolimits}

\newtheorem{utv}{Proposition}
\newtheorem{lem}{Lemma}

\title {Probabilistic properties of topologies of finite metric spaces' minimal fillings.}
\author {Vsevolod Salnikov}
\date{}
\begin{document}

\maketitle

\begin{abstract}
In this work we provide a way to introduce a probability measure on the space of minimal fillings of finite additive metric spaces as well as an algorithm for its computation. The values of probability, got from the analytical solution, coincide with the computer simulation for the computed cases. Also the built technique makes possible to find the asymptotic of the ratio for families of graph structures.

\end{abstract}

\section {Introduction}

In the middle of the 17th century Pierre de Fermat stated a question, how to find a point in a plane, such that a sum of the distances to three fixed points has a minimal value. This point is now called Fermat point of the triangle. Later on the problem has been generalized and now there are several different formulations, united under the common name of a Steiner problem.

One of the most common versions of the Steiner problem is the following. Let $X$ be a finite set of points in a plane. We define a {\em connecting tree} $T(X)$ as a finite set of straight segments, whose set of end points contains $X$, and for any pair of points from $X$ there exists one and only one polygonal line connecting them and built from the chosen segments. For each connecting tree we can calculate the sum $l(T(X))$ of the Euclidean lengths of all included line segments. If there exists such a connecting tree $T^*(X)$, that for any connecting tree $T(X)$ the inequality $l(T^*(X)) \le l(T(X)))$ holds, then this tree is called a {\em (Euclidean) Steiner minimal tree}. Unfortunately the Steiner problem is NP-hard. Moreover, this problem has been in one of the first lists of NP-hard problems published by Richard M. Karp in 1972  \cite{NP}. Garey, Graham and Johnson \cite{DESMT} have switched from the continuous space to the discrete one and proved that it is also NP-hard.

In one of his works Gromov \cite{Gromov} has defined a concept of a minimal filling. Let $M$ be a smooth closed manifold with a distance function $\rho$. Let's consider all films $W$, spanning $M$, i.e. compact manifolds with boundary equal to $M$. Let's define a distance function $d$ on $W$, which doesn't reduce the distance between points from $M$, i.e. $\rho(P,Q)\le d(P,Q)$ for all $P$ and $Q$ from $M$. Such a metric space $W = (W, d)$ is called a {\em filling} of the metric space $M = (M, \rho)$. Gromov problem consists of the description of the infimum of the fillings volumes and such spaces $W$, called {\em minimal fillings}, on which the infimum is achieved at. This problem is another example of an optimization problem in geometry.

A.Ivanov and A.Tuzhilin \cite{Sledstv} have suggested to consider a minimal filling problem for the finite metric spaces as a generalization of both Steiner problem and Gromov's minimal filling problem.

For a fixed weighted graph $G$ with a weight function $w$, let's define a distance function $d_w(p,q)$ between it's vertices $p$ and $q$ to be equal to the minimum of the weights of the paths, connecting these vertices in $G$, and equal to infinity if there are no such paths.

For a finite metric space $(M,\rho)$, graph $G=(V,E)$ with non-negative weights of edges we call {\em connecting} $M$, if $M \subset V$ and for each pair of vertices from $M$, there exists a path between them in the graph. Let's define a 
{\em filling} of a finite metric space $(M,\rho)$ as a weighted graph  ${\EuScript G} = (G, w)$ connecting $M$ and such that for each $p, q \in M$ we have $\rho(p, q) \le d_w(p, q)$. Then graph $G$ is called the {\em type} of the filling or it's  {\it topology}. A value $\mf((M,\rho)) = \inf w({\EuScript G})$ among all fillings ${\EuScript G}$ of the space $(M,\rho)$, we call a {\em weight of minimal filling}, and every filling  ${\EuScript G}$, such that $w({\EuScript G}) = \mf((M,\rho))$ пїЅ {\em minimal filling} of the space $(M,\rho)$. Subset of the vertices $\partial G \subset V$, corresponding to $M$ is called {\em boundary}, and it's elements ---  {\em boundary vertices}. In \cite{Sledstv} it is shown, that without loss of generality it can be assumed, that $G$ is a tree and it's boundary contains only vertices with degree 1. 
It is very important to keep in mind, that we look not only for a graph, but also for a {\em boundary map} --- a bijection from the boundary vertices set $\partial  G$ to $M$. 

For specific cases the minimal filling problem has been solved. For example, for three points $p_1, p_2, p_3$  with distances $\rho_{12}, \rho_{23}, \rho_{13}$ between them   it is easy to see, that a minimal filling contains one extra vertex $x$, connected to each point where weight of the edge $xp_i$ is the following 
$$w(xp_i) = \frac {\rho_{ij} + \rho_{ik}  - \rho_{jk}}{2}$$
Later on, the general formula for the minimal filling weight has been found \cite{MFform}, but it doesn't simplify the search of the minimal filling, because it contains exponential-size check of all {\it a priori} possible types of the filling with binary tree structure.

If the type is fixed, then the minimization problem among the fillings with this type (called  {\em minimal parametric filling}), can be easily reduced to a linear programming problem, which can be efficiently solved. Unfortunately the set of possible types grows exponentially with the number of points in the metric space, that is why it is not possible to handle each of them. This makes interesting to look for an approximate solution.

In this work a new method to find an approximate solution is suggested. More precisely, we will study, which topologies are more probable and after that the approximate solution will be found only among the minimal parametric fillings with one of the most probable topologies. Currently we restrict ourselves to additive metric spaces, because there is more information about minimal fillings structure for these metric spaces, but the approach apparently can be generalized for arbitrary metric spaces as well as for other optimization problems.

\section {Additive space case}

Let $(M,\rho)$ be a finite metric space. It is called {\em additive}, if there exists a  {\em generating tree}, i.e. weighted tree with a set of boundary vertices equal to $M$, and  $\rho$-distance between any two points $p,q \in M$ equals to $d_w(p,q)$ .

Currently we will take into account only non-negative weights, while similar questions can be stated for generalized additive spaces  \cite{Zahar}, where negative weights can occur.

It is known, that for an arbitrary additive space the generating tree without degenerate edges is unique \cite{Unique}. And this generating tree as well as all other generating trees, is a minimal filling for this metric space \cite{Sledstv}.

In accordance with \cite{Sledstv} we call {\em mustache} a subgraph, which, for some $k$, consists of $k$ adjacent edges, each of them is incident to a boundary vertex, and degree of the common vertex equals to $k + 1$.

Let's define and calculate the value of a probability that the graph of a minimal filling of the additive space has a topology $G$. This value will be denoted by $\P( G)$. Let's identify points in $M$ with positive integers $\{1,\ldots,n\}$. With a fixed topology we have a map $T_ G \colon w \mapsto \rho$, from the weight distributions to $n \times n$ distance-matrices $\rho$, where $\rho_{p,q} = d_w(p, q)$. From the following lemma we will see that there exists an inverse map, which gives a weight distribution from the additive space distance-matrix and the generating tree topology.

\begin{lem}
Let's fix a tree $G$ without non-boundary vertices with degree 2 (the limitation doesn't change the generality, because such vertex with both incident edges can be always replaced by one edge), and suppose, that there exists a non-negative weight distribution $w$ making $G$ a generating tree for the space $M$, then the weight distribution $w$ can be uniquely found from the distance-matrix.
\end{lem}
\begin{proof}
Let's prove by induction by the number of edges in the tree $G$. If there is only one edge, then it is necessary, that it's weight is equal to the distance between vertices it is connecting. Suppose that the statement is proved for all trees $G$ with not more than $k$ edges, let's consider a tree with $k+1$ edge. We can find a vertex $p$ of degree 1. Let $q$ be a unique vertex, connected with $p$ by an edge. If degree of $q$ is less than 3, then $q$ is boundary vertex, and it is necessary that the weight of the edge connecting $p$ and $q$ is equal to the distance between the vertices. If the degree is more than 2, there exists at least 2 boundary vertices $q_1$ and $q_2$, different from $p$, such that a path, connecting $q_1$ and $q_2$, passes through $q$. In this case the weight of the edge  connecting $p$ and $q$ is necessary equal to $(\rho_{p,q_1} + \rho_{p,q_2} - \rho_{q_1,q_2}) / 2$. All other weights can be found if we will work with the initial tree without vertex $p$, it's incident edge and taken $q$ as a boundary vertex (if it isn't the case, then we put the distance from $q$ to any boundary vertex $q_3$ equal to the difference between $\rho_{p,q_3}$ and the found weight of the edge).
\end{proof}

Let's fix a norm on the distance-matrices space. For further calculations it is convenient to take the  $\ell_1$ norm of elements above the diagonal (on the diagonal there are zeros, the matrix is symmetric, thus this norm is a half of the standard $\ell_1$-norm of the matrix). 

We will define the probability $P(G)$ as a ratio of the measure of all additive spaces with topology $G$ to the measure of all additive spaces, having the same number of vertices. It is possible to show, that for each additive space we can find a generating tree, which is a binary tree (the degree of a vertex can be 1 or 3). That's why among the possible types $G$ we take into account only binary trees. In this case some weights can have zero values and there might exist different topologies of a generating tree, but we define the measure in a way, that the measure of additive spaces having zero-weights in the generating tree is zero.

Let us now describe the measure. We need to describe an additive space of type $G$ with a fixed boundary map (correspondence between vertices with degree 1 of the tree $G$ and the set  $M$). Let's enumerate all edges of the chosen topology and denote their weights by $x_1, x_2, ..., x_m$ (the quantity $m$ of edges in a binary tree with $n$ boundary vertices is equal to $2n -3$). Having the set $\{x_i\}$ we can build a distance-matrix $\rho$ between points from $M$. As all $x_i \ge 0$, all elements in the matrix $\rho$ are non-negative and hence the $\ell_1$ norm of the matrix is equal to the sum of it's elements, $\|\rho\|_1 =  \sum q_i x_i$, where $q_i$ is the number of occurrences of the edge $i$ in the paths from vertex $j$ to vertex $k$, for all $j,k$ : $j < k$. Let's note that each $q_i$ is equal to the product of the number of boundary vertices from one side of the edge and the number of boundary vertices from another side. 

%------------------
\begin{utv}
The map $T_ G$ constructed above  is linear.
\end {utv}
\begin {proof}
For each pair of vertices  $p,q$ there exists a unique path in the tree $G$, that connects them. Then the corresponding distance matrix component is equal to the sum of weights of edges $\{r_{i_k}\}_{k=1}^t$, making this path. Let's denote by $\rho_{w} $ the distance-matrix, corresponding to the weight distribution $w$, where $w_{i_k}$ is the weight of the edge $r_{i_k}$.  Let $w = w^1 + w^2$, then
$$\rho_{w} (p,q) = \sum_{k=1}^{t} w_{i_k} = \sum_{k=1}^t (w^1_{i_k} + w^2_{i_k}) = \sum_{k=1}^{t} w^1_{i_k} + \sum_{k=1}^{t} w^2_{i_k} = \rho_{w^1} (p,q)+\rho_{w^2} (p,q)$$
For the case $w=\lambda w^1$ the proof is similar.
\end{proof}

\begin{lem}
Let $G$ be a binary tree, then the rank of $T_ G$ is maximal.
\end {lem}
\begin {proof}
Our tree $G$ satisfies conditions of the Lemma 1 and thus the weight distribution is unique. Suppose that the rank is not maximal, then  $\Ker T_G$ is non-trivial. Let's take an arbitrary weight distribution $\{w_i\}$ and a non-zero element $\{w^0_i\}$ from the kernel. Let $\mu = 2 \min (\min_i (w_i), \min_i(w_i+w^0_i))$. We can build two positive weight distributions: $\{w^1_i = w_i + |\mu|\}, \{w^2_i = w_i + w^0_i + |\mu|\}$, which differ by the element from the kernel. According to the linearity, they have the same image, but it contradicts to the uniqueness of the weight distribution.

\end{proof}

Now it is clear that the image of the positive orthant after a map $T_G$ is a convex $m$-dimensional cone, where $m$ is a number of edges in the selected topology.

\begin{utv}
The intersection of the $T_G$-image of the positive orthant with a plane $\|\rho\|_1 = 1$ is a $T_G$-image of an $m$-dimensional simplex.
\end {utv}
\begin {proof}
This simplex can be easily constructed. It's vertices are weight distributions $w^i$ with only one non-zero component $w^i_i = 1 / q_i$. This simplex consists of the points of the form $w = \sum w^i t_i$, where $t_i \in [0,1]$ and $\sum \limits_i t_i = 1$. According to the linearity and the selected norm, we get $\|\rho_w\|_1 = \sum t_i \|\rho_{w^i} \|_1= \sum t_i = 1$ (sums are among all edges). 

Vice versa, for each point $\rho$  from the image of the orthant, hence $T_G(w)$ belongs to the plane, such that $\|\rho\|_1 = 1$, consider its pre-image $\bar w$. The weight distribution $\bar w$ can be represented as $\bar w = \sum k_i w^i$, because $w^i$ form a basis, and according to the choice of $w^i$ and the fact, that  components of the $\bar w$ are non-negative, $k_i$ are also non-negative and $1=\|\rho\|_1 = \sum k_i$.
\end {proof}

Let's denote by $\p_{ G}$ the probability that a minimal filling has topology $ G$ with enumerated boundary vertices, i.e. the boundary mapping is fixed. 
We assume, that $\p_{ G}$ is proportional to the volume of the intersection of the positive orthant image and a ball $\|\rho\|_1 \le 1$. This choice is, of course, not unique, but very natural, as it is a standard measure in the vector space. 

Taking different boundary maps,  {\it a priori} we can get different images after the map $T_G$. But in the meanwhile, it is clear, that if chosen boundary maps differ by the symmetry of the graph, then we will get the same image (weight of the edges will switch together with the edges). Thus automorphisms of the graph generate symmetries of the image after the map $T_G$. Later on we will say that topologies are different if the images of corresponding maps are different. The number of elements in the group of automorphism of the tree $G$, containing such elements, that generate the isometry of the corresponding additive space, we will denote by $\simm ( G)$. After the defined notations, it is clear that the probability $\P(G)$, that a minimal filling will have topology $G$ (with some boundary map) is defined by the following formula
$$\P( G) = c \cdot \p_{G} \cdot n! / \simm ( G)= C \cdot \vol (\Im T_{ G} \cap B_1(0)) \cdot n! / \simm ( G),$$
where $B_1(0)$ --- a unit ball in the selected norm and constants $c$ and $C$ don't depend on $ G$.

Thus $$\P( G_1) / \P( G_2) = \frac {\p_{G_1} }{\p_{G_2} } \frac {\simm ( G_2)}{\simm ( G_1)}$$

\begin{lem}
If for two different topologies and some weight distributions we get the same additive space, then it is necessary, that there are zeros in weight distributions, i.e. if $T_{G_1}(w_1)=T_{G_2}(w_2)$, then at least one from the weight distributions $w_1, w_2$ contains zero components.
\end{lem}
\begin{proof}
A generating tree without degenerate edges is unique, thus if there are two different topologies, then there are degenerate edges.
\end{proof}

Hereby we have, that the volume of the intersection of the images for two different topologies is zero.
  
Absolutely similar to the proposition 2 we can prove the following
\begin{utv}
The intersection of the image of the positive orthant after a map $T_G$ with a ball $\|\rho\|_1 \le 1$ --- image of a $m$-dimensional simplex after a map $T_G$.
\end {utv}
\begin {proof}
The only difference is that we need to add a vertex $w = (0,пїЅ,0)$ to the basis. The rest is the same.
\end {proof}

There has left only to find a way, how to calculate volumes of these simplexes. According to the fact, that the image of the simplex is a simplex, and we know coordinates of its vertices, then the volume $V$ can be calculated with a help of a matrix $W$, built from the vectors
$$W_i = V_i - V_0,$$
where the vector $V_i$ --- consists of the elements of the upper triangle of the distance matrix $T_G(w^i)$, enumerated in some way, the same for all i. In this notations
$$V = \frac {|\det W W^T|^{\frac{1}{2}}}{m!}$$ 
Moreover, we have a vertex  $w = (0,пїЅ,0)$, and can give it a zero number, then $W_i=V_i$.

Let $Q=W W^T$. In the selected numeration $Q_{ij} = <V_i, V_j>$, where $< , >$ denotes a scalar product of vectors. 
A coordinate of the vector $V_i$ is either  $1/q_i$, if $i$-th edge belongs to the path connecting the corresponding pair of boundary vertices or zero otherwise. Thus $Q_{ij}$ is equal to the number of unordered pairs of vertices, such that both i-th and j-th edge belongs to the path, connecting a pair of boundary vertices, divided by $q_iq_j$ (here we see, that the result doesn't depend on the way we have enumerated coordinates of vectors $V_i$, which choice hasn't been described above).

\section {Example}

The minimal number of points in $M$, such that there exists two different (not isomorphic) binary tree topologies, is 6. Let's enumerate edges of topologies as follows:

\includegraphics[width=50mm]{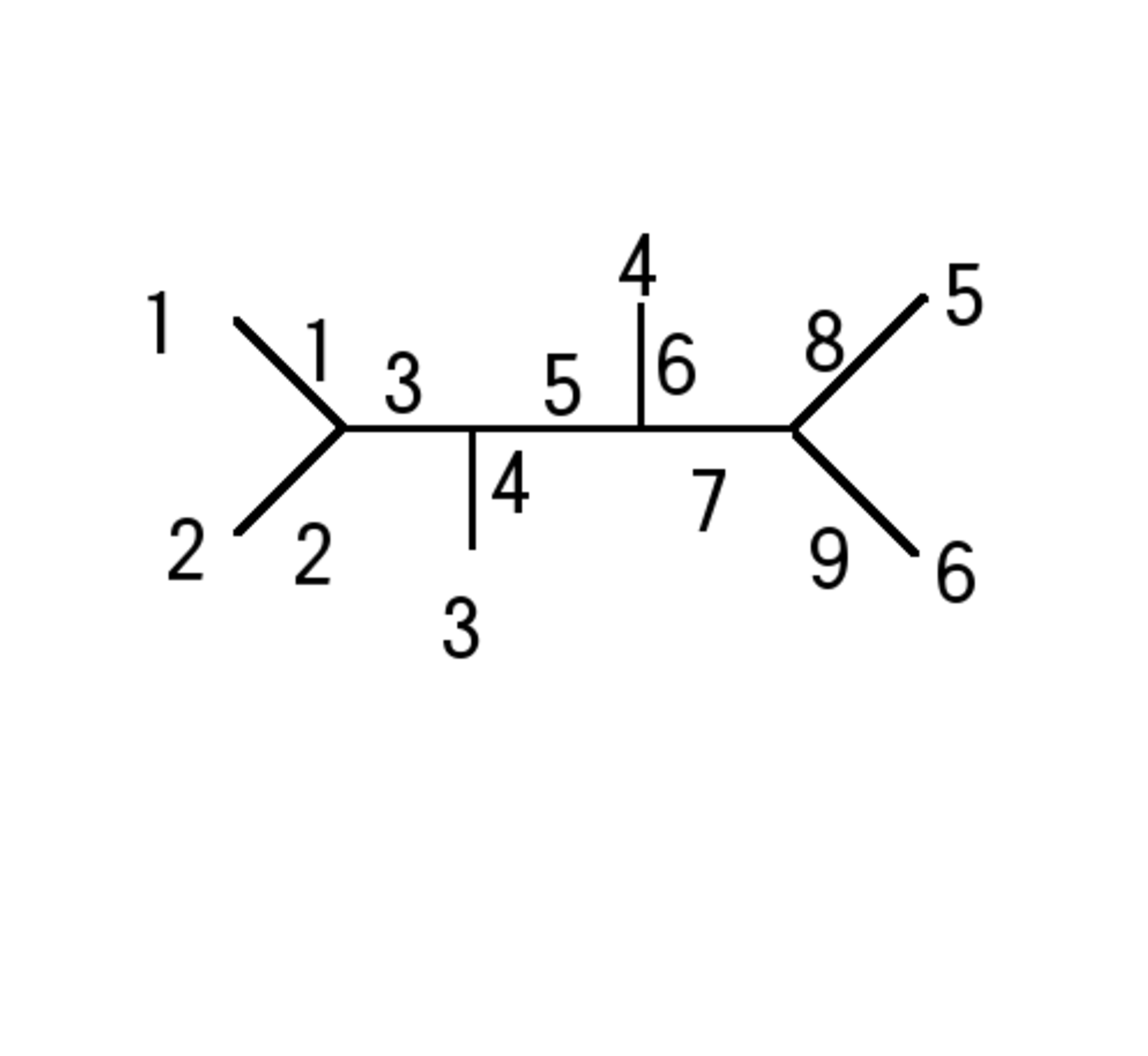}
\includegraphics[width=50mm]{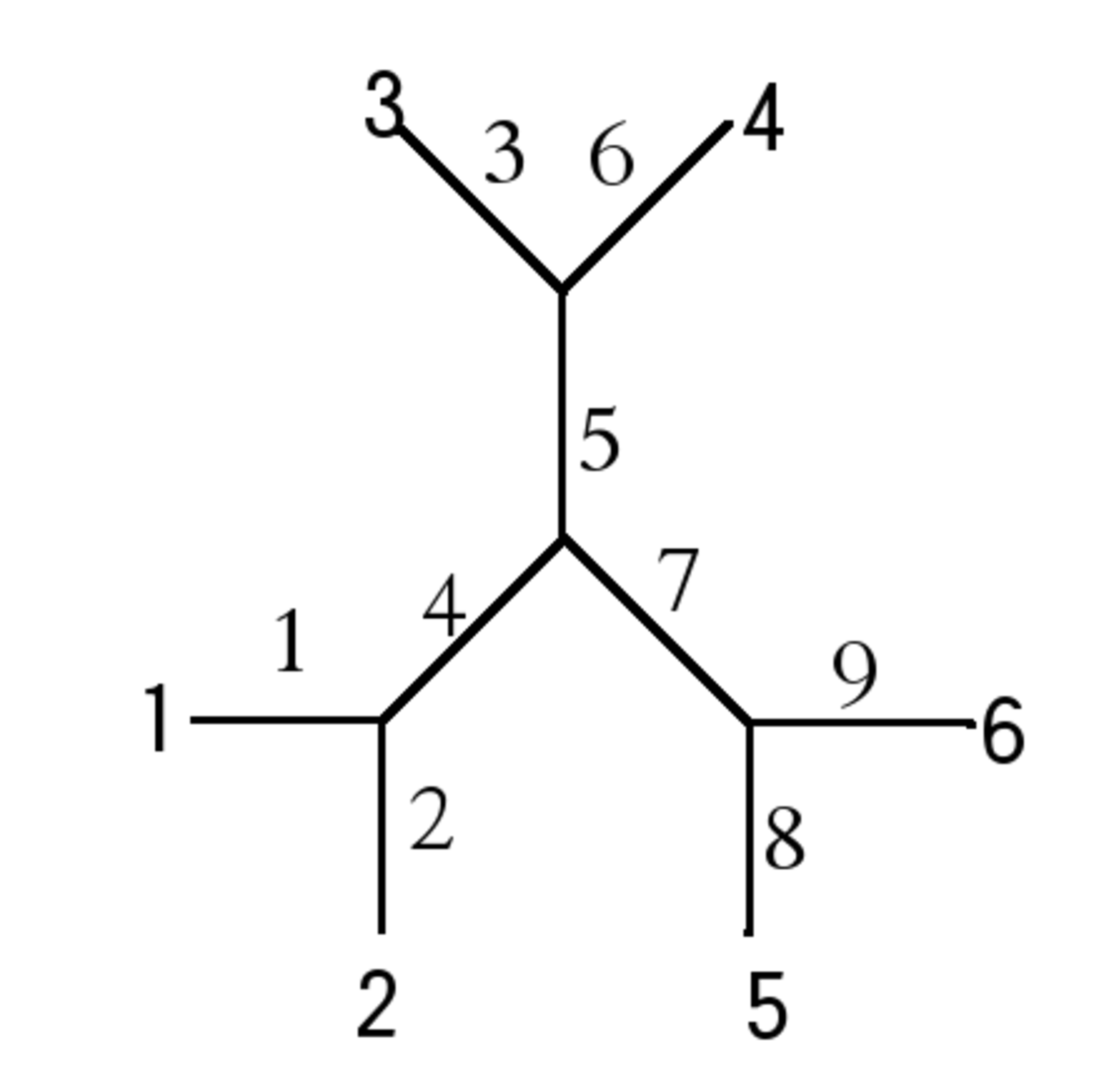}

Here $m=9$. Let's denote by $q^k$ a vector consisting of the consecutive values $q_i$  for the $k$-th topology. Similarly  $Q^k$ is the matrix $Q$ for the $k$-th topology.
$$q^1=(5,5,8,5,9,5,8,5,5)$$
$$q^2=(5,5,5,8,8,5,8,5,5)$$
$$ Q^1 = \begin{pmatrix}
                              1/5&1/25 & 1/10& 1/25& 1/15& 1/25&1/20 &1/25 &1/25\\
                              1/25& 1/5& 1/10& 1/25& 1/15& 1/25&1/20 &1/25 &1/25\\
                              1/10& 1/10& 1/8& 1/20& 1/12& 1/20&1/16 &1/20 &1/20\\
			       1/25& 1/25& 1/20&1/5& 1/15& 1/25&1/20 &1/25 &1/25\\
                              1/15& 1/15& 1/12&1/15 &1/9& 1/15&1/12 &1/15 &1/15\\
                              1/25& 1/25& 1/20&1/25 &1/15 &1/5&1/20 &1/25 &1/25\\
                              1/20& 1/20& 1/16&1/20 &1/12 &1/20 &1/8 &1/10 &1/10\\
                              1/25& 1/25& 1/20&1/25 &1/15 &1/25 &1/10 &1/5 &1/25\\
                              1/25& 1/25& 1/20&1/25 &1/15 &1/25 &1/10 &1/25 &1/5
                             \end{pmatrix}$$ 
$$ Q^2 = \begin{pmatrix}
                             1/5& 1/25 &1/25 &1/10 &1/20 &1/25 &1/20 &1/25 &1/25\\
                             1/25& 1/5 &1/25 &1/10 &1/20 &1/25 &1/20 &1/25 &1/25\\
                             1/25& 1/25& 1/5 &1/20 &1/10 &1/25 &1/20 &1/25 &1/25\\
                             1/10& 1/10& 1/20& 1/8 &1/16 &1/20 &1/16 &1/20 &1/20\\
                             1/20& 1/20& 1/10&1/16& 1/8  &1/10 &1/16 &1/20 &1/20\\
                             1/25& 1/25& 1/25&1/20& 1/10& 1/5  &1/20 &1/25 &1/25\\
                             1/20& 1/20& 1/20&1/16& 1/16&1/20 &1/8   &1/10 &1/10\\
                             1/25& 1/25& 1/25&1/20& 1/20&1/25 &1/10 &1/5   &1/25\\
                             1/25& 1/25& 1/25&1/20& 1/20&1/25 &1/10 &1/25 &1/5
                             \end{pmatrix}$$ 

The determinants of these matrices are equal to   $\frac{4}{9} \cdot 5^{-12}$ and $\frac{1}{2} \cdot 5^{-12}$ correspondingly. 
$$\simm(G_1)=8,\simm(G_2)=48.$$ Thus $\p(G_1) / \p(G_2) = \frac{16}{3}$, which gives $$\p(G_1) = 16/19,\p(G_2) = 3/19.$$

\section {Calculation of the determinants in general case}

Let's consider a binary tree with  $n$ boundary vertices. Let's fix one of the vertices with degree 3 and call it {\em center}. Suppose that edges are enumerated in some way. As before, we will say that edges $i$ and $j$ {\em belong to one branch}, if there exists such a vertex, that a path from the center to this vertex contains both of edges. We will denote by the {\em level } $l(i)$ of the edge with a number $i$ the number of edges in the path, connecting center with the closest vertex, incident to the edge $i$.

\begin {utv}
For any binary tree with  $n$ boundary vertices and a fixed center, the following holds
$$\det Q =\frac{4}{\prod\limits_i (n-n(i))^2 }\cdot\prod_{(j,k)}\Big[4\Big(\frac {n}{n(k)+n(j)}-1 \Big)\Big], $$ where $n(i)$ ( $i$ is an edge) is the number of boundary vertices such that a path from the center to this vertex contain the edge $i$. Here the first product is taken over all the edges $i$, and the second product is taken among all pairs $(j,k)$ such that edges $j$ and $k$ are adjacent and $l(j)=l(k)\ne 0$.

\end{utv}

\begin {proof}
Let's note that $q_i$=$n(i)(n-n(i))$. If edges $i$ and $j$ belong to one branch, then without loss of generality $l(i)\le l(j)$ and $Q_{ij}=\frac {n(j)(n-n(i))}{q_{i}q_{j}}$. If edges  $i$ пїЅ $j$ belong to different branches (don't belong to one branch), then $Q_{ij}=\frac{n(i)n(j)}{q_{i}q_{j}}$.

Suppose that edges $i$,$j$ belong to one branch and $l(j)=l(i)+1$. Let's consider the possible placements of an extra edge $k$ with fixed $i$ and $j$:

$(1)$.  $k=j$

$(2)$.  $k$ belongs to the common for $i$ and $j$ branch and $l(k) \le l(i)$

$(3)$. $k$ belongs to the common for $i$ and $j$ branch and  $l(k) > l(j)$

$(4)$. $k$ belongs to one branch with $i$, but not with $j$ and $l(k) = l(j)$

$(5)$. $k$ belongs to one branch with $i$, but not with $j$ and$j$ пїЅ $l(k) > l(j)$

$(6)$. $k$ doesn't belong to one branch with $i$.

Let's calculate the value $R_{ijk}=(n-n(i))Q_{ik}-(n-n(j))Q_{jk}$ for some of mentioned above cases. 

$(1)$.  $R_{ijj}=(n-n(i))Q_{ij}-(n-n(j))Q_{jj}=\frac {(n-n(i))^2 n(j)}{q_{i} q_{j}} - \frac {(n-n(j))^2 n(j)}{q_{j}^2}=\frac {n-n(i)}{n-n(j)}  \frac{1}{n(i)}-\frac {1}{n(j)}$

$(2)$.  $l(k)\le l(i)< l(j)$ пїЅ $R_{ijk}=\frac{(n-n(i))n(i)(n-n(k))}{q_{i}q_{j}} -\frac{(n-n(j))n(j)(n-n(k))}{q_{j}q_{k}}=0 $

$(4)$. $R_{ijk}=\frac {(n-n(i))^2 n(k)}{q_{i} q_{k}} - \frac  {(n-n(j)) n(j)n(k)}{q_{j} q_{k}}= \frac {n-n(i)}{n(i)(n-n(k))} - \frac{1}{n-n(k)}$

$(6)$. $R_{ijk}=\frac{(n-n(i)) n(i)n(k)}{q_{i} q_{k}}-\frac {(n-n(j)) n(j)n(k)}{q_{j} q_{k}}=0$

For cases  $(3)$ and $(5)$  $l(k)>l(j)$  and a pair $(k,i)$ belongs to one branch. It is sufficient for us.

{\bf Now everything is ready to calculate $\det Q=\det ( Q_{ij})$.}

The center is incident to 3 edges. If $r$ is one of these edges, then the maximal subtree consisting of the edges lying on the same branch with r is referred as the {\em main branch} corresponding to $r$. 

{\bf Step $1$}. Let's take one of main branches, starting from the center. Let  $j$ be the number of an edge with maximal level in this branch. There can be two cases:

a). $l(j)=0$ and then $j=r$

b). $l(j)\ne 0$ and then  we know from the binary structure that there are exactly 2 edges adjacent to $j$: say, $i$ and $k$. Moreover  $l(i)=l(j)-1$, $l(k)=l(j) $ and pairs $(i,k)$ and $(j,i)$ belong to one branch, but not the pair $(j,k)$. Further we will call edges $j,k$ {\em examined}.

{\bf Step $2$}. Let's make the following elementary operations with rows of the matrix $Q$:

$1) $. From the $j$-th row subtract the $i$-th row, multiplied by the non-zero coefficient $\frac {n-n(i)}{n-n(j)}$. According to the maximality of level of the edge $j$, there are no edges from the same branch with bigger level. After this subtractions, for the element with number  $t$ in the $j$-th row we have $\hat Q_{jt}=\frac{-R_{ijt}}{n-n(j)}$. Thus non-zero values will have only elements $\hat Q_{jj}$ and $\hat Q_{jk}$.

$2)$. Similarly from the $k$-th row subtract the $i$-th row, multiplied by $\frac {n-n(i)}{n-n(k)}$. In the $k$-th row of the obtained matrix non-zero values will have only elements $\hat Q_{kj}$ and $\hat Q_{kk}$.

{\bf Step $3$}. According to the form of rows with numbers $k$ and $j$ we have

$$\det Q=\det \begin {pmatrix}
 \hat Q_{kk}&\hat Q_{kj}\\
\hat Q_{jk} &\hat Q_{jj} \\ 
\end{pmatrix} \cdot \det Q_1,$$
 where a matrix $Q_1$ is obtained from the matrix $Q$ by throwing out rows and columns with numbers $k$ and $j$.

{\bf Step $4$.}
So, we have:
$$\det \begin {pmatrix}
 \hat Q_{kk}&\hat Q_{kj}\\
\hat Q_{jk} &\hat Q_{jj} \\ 
\end{pmatrix}=$$ $$=\det \begin {pmatrix}
\frac {1}{n(k) (n-n(k))}-\frac{n-n(i)}{n(i)(n-n(k))^2}&\frac{1}{(n-n(j))(n-n(k))}(1-\frac{n-n(i)}{n(i)})\\
\frac{1}{(n-n(j))(n-n(k))}(1-\frac{n-n(i)}{n(i)})&\frac {1}{n(j) (n-n(j))}-\frac{n-n(i)}{n(i)(n-n(j))^2}\\
\end{pmatrix}=$$ $$=\frac{1}{(n-n(j))^2 (n-n(k))^2} \Big [\Big(\frac{n-n(k)}{n(k)} -\frac {n-n(i)}{n(i)}\Big)\Big(\frac {n-n(j)}{n(j)}-\frac {n-n(i)}{n(i)}\Big)-\Big(1-\frac {n-n(i)}{n(i)}\Big)^2 \Big]=$$ $$=\frac {1}{(n-n(j))^2 (n-n(k))^2} \Big[\Big(\frac{n}{n(k)}-\frac{n}{n(i)}\Big)\Big(\frac{n}{n(j)}-\frac{n}{n(i)}\Big)-\Big(2-\frac{n}{n(i)}\Big)^2\Big]=$$ $$=\frac {1}{(n-n(j))^2 (n-n(k))^2} \Big[\Big(\frac{n}{n(k)}-\frac{n}{n(k)+n(j)}\Big)\Big(\frac {n}{n(j)}-\frac{n}{n(k)+n(j)}\Big)-\Big(2-\frac{n}{n(k)+n(j)}\Big)^2\Big],$$ where the last equality holds, because  $n(i)=n(k)+n(j)$ (the tree is binary).

{\bf Step $5$.} Let's consider the edge $j_1$ from the same branch having the maximal level among the non-examined edges.
If  $l(j_1)>0$, then we repeat Steps 1-4 (if $l(j_1)<l(j)$, then in the Step 2 there might be extra non-zero values in 'thrown out' columns, but they don't influence the determinant value).

{\bf Step $6$.} Repeat Step 5 while there are edges with non-zero level.

{\bf Step $7$.}  Repeat Steps 1-6 for other main branches.

{\bf Step $8$.} Suppose that after Steps  1-7 we have got the matrix $\bar Q$. It is obtained from the initial matrix by stroking out all the rows and columns, but not the ones, corresponding to the edges incident to the center (let them have the numbers $c_1$,$c_2$,$c_3$). Then

$$\det \bar Q=
\det \begin {pmatrix}
\frac {1}{n(c_1) (n-n(c_1))}&\frac{1}{(n-n(c_1))(n-n(c_2))}&\frac{1}{(n-n(c_1))(n-n(c_3))}\\
\frac {1}{(n-n(c_1)) (n-n(c_2))}&\frac{1}{n(c_2)(n-n(c_2))}&\frac{1}{(n-n(c_2))(n-n(c_3))}\\
\frac {1}{(n-n(c_1)) (n-n(c_3))}&\frac{1}{(n-n(c_2))(n-n(c_3))}&\frac{1}{n(c_3)(n-n(c_3))}\\

\end{pmatrix}=$$ 

$$=\frac {1}{\prod\limits_i (n-n(c_i))} 
\begin {pmatrix}
\frac {1}{n(c_1) }    &  \frac{1}{n-n(c_2)}  &  \frac{1}{n-n(c_3)}\\
\frac {1}{n-n(c_1)}  &  \frac{1}{n(c_2)}     &  \frac{1}{n-n(c_3)}\\
\frac {1}{n-n(c_1)}  &  \frac{1}{n-n(c_2)}  &  \frac{1}{n(c_3)}\\
\end{pmatrix}=$$ 

$$=\frac {1}{\prod\limits_i (n-n(c_i))^2} 
\begin {pmatrix}
\frac {n-n(c_1)}{n(c_1)} & 1 & 1 \\
1 & \frac {n-n(c_2)}{n(c_2)} & 1 \\
1 & 1 & \frac {n-n(c_3)}{n(c_3)}\\
\end{pmatrix}=$$ 
 
$$=\frac {1}{\prod\limits_i (n-n(c_i))^2} \cdot \Big[\prod\limits_j \frac{n-n(c_j)}{n(c_j)}-\frac{n-n(c_1)}{n(c_1)}-1\cdot \Big(\frac{n-n(c_3)}{n(c_3)}-1 \Big)+1\cdot \Big(1-\frac{n-n(c_2)}{n(c_2)}\Big)\Big]=$$ 
$$ = \frac{1}{\prod\limits_i (n-n(c_i))^2}\Big[\Big(\frac{n}{n(c_1)}-1\Big)\Big(\frac{n}{n(c_2)}-1\Big)\Big(\frac{n}{n(c_3)}-1\Big)-\sum_j \frac{n}{n(c_j)}+5\Big]=$$ 
$$=\frac {1}{\prod\limits_i (n-n(c_i))^2}\Big[\frac{n^3}{\prod\limits_j n(c_j)}-\frac{n^2(n(c_1)+n(c_2)+n(c_3))}{\prod\limits_j n(c_j)}+\sum_j \frac{n}{n(c_j)}-1-\sum_j \frac{n}{n(c_j)}+5\Big]=$$ $$=\frac{4}{\prod\limits_i (n-n(c_i))^2},$$ where we have used $n=n(c_1)+n(c_2)+n(c_3)$.

{\bf Step $9$.} Hereby the determinant can be calculated:  
$$\det Q=\frac{4}{\prod\limits_{l=1}^3 (n-n(c_l))^2}\prod\limits_{(j,k)}\Big[\frac{1}{(n-n(j))^2(n-n(k))^2}\cdot $$ $$\cdot \Big(\Big(\frac{n}{n(k)}-\frac{n}{n(k)+n(j)}\Big) \Big(\frac{n}{n(j)}-\frac{n}{n(k)+n(j)}\Big)-\Big(2-\frac {n}{n(k)+n(j)}\Big)^2\Big)\Big]=$$ $$=\frac{4}{\prod\limits_i (n-n(i))^2 }\cdot\prod_{(j,k)}\Big[n^2 \Big(\Big(\frac {1}{n(k)}-\frac {1}{n(k)+n(j)}\Big) \Big(\frac{1}{n(j)} - \frac{1}{n(k)+n(j)}\Big)-\Big(\frac{2}{n}-\frac{1}{n(k)+n(j)}\Big)^2\Big)\Big]=$$
$$ =\frac{4}{\prod\limits_i (n-n(i))^2 }\cdot\prod_{(j,k)}\Big[4\Big(\frac {n}{n(k)+n(j)}-1 \Big)\Big], $$ where pairs $(j,k)$ are such that edges  $j$ and $k$ are adjacent and $l(j)=l(k)\ne 0$.

\end {proof}

{\bf Remark}. The center can be chosen in a way that is more comfortable for calculations and the final result doesn't depend on the choice.

\section {The ratio of probabilities for fixed types}

Now we can use the algorithm described above  to find the asymptotic of the ratio for fixed topologies. Let's fix an even $n$ and choose topologies, which have a same number of symmetries:
\newline
\includegraphics[width=70mm]{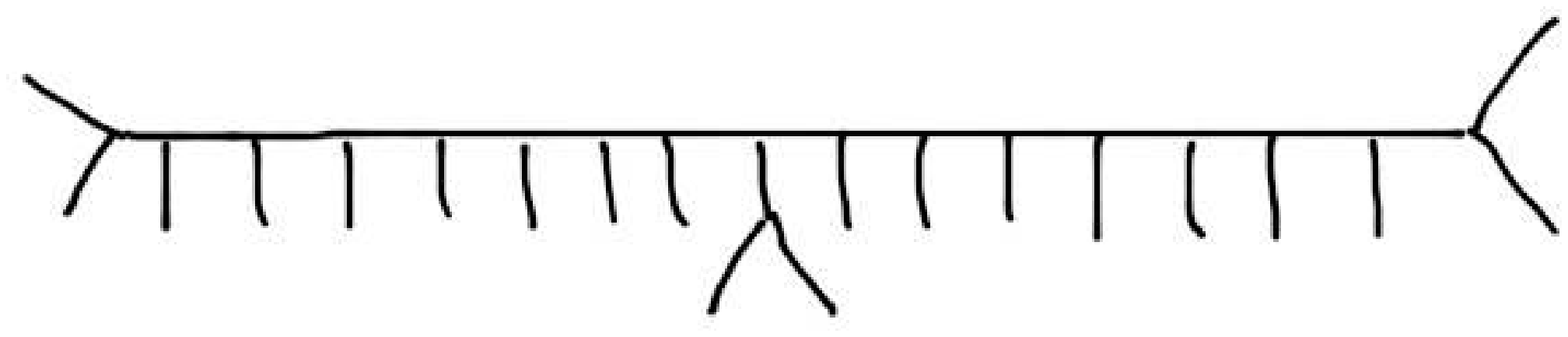}
\newline
and
\newline
\includegraphics[width=70mm]{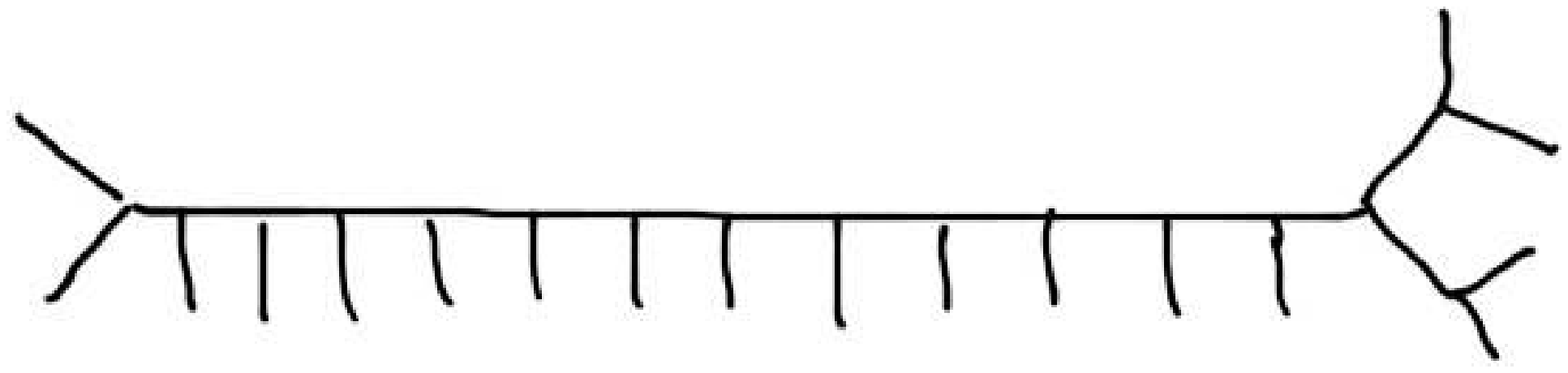}

To calculate the determinant of the matrix  $Q_1$ (for the first topology) we can choose a center in a way, that it cuts the tree into a subtree, containing 3 edges, and two similar subtrees, containing  $(n-2)/2$ boundary vertices each. For each of the similar branches 
$$\prod_{(j,k)}\Big[4\Big(\frac {n}{n(k)+n(j)}-1 \Big)\Big] =  \Big[ 4^{n/2-2}\frac {(n-2)!}{\big(\frac{n-2}{2}\big)!\big( \frac{n}{2}\big)!}\Big],$$
because without loss of generality we can say, that $n(j)=1$ and $n(k)$ takes all values from 1 to $n/2-2$. Thus

$$\det Q_1 =\frac{4}{\prod\limits_i (n-n(i))^2 }\cdot \Big[ 4^{n/2-2}\frac {(n-2)!}{\big(\frac{n-2}{2}\big)!\big(\frac{n}{2}\big)!}\Big]^2\Big[4\Big(\frac{n}{2} - 1 \Big) \Big]=$$
$$=\frac {4^{n-2}}{(n-1)^{2n}(n-2)^2\Big[ \frac {(n-2)!}{(\frac {n}{2} )!}  \Big]^4}\cdot \Big[ \frac {(n-2)!}{\big(\frac{n-2}{2}\big)!\big(\frac{n}{2}\big)!}\Big]^2\Big[\frac{n}{2} - 1 \Big]=$$
$$=\frac {4^{n-2}}{(n-1)^{2n}(\frac {n}{2}-1)!^2(n-2)^2\Big[ \frac {(n-2)!}{(\frac {n}{2})!}  \Big]^2}\Big[\frac{n}{2} - 1 \Big]=\frac {4^{n-3} n^2}{2(n-1)^{2n}(n-2)(n-2)!^2}$$

For the second topology the choice of the center will be the following: one branch is the same as one of similar subtrees from the previous case, the second branch contain only one edge, and all the rest is in the third branch. For the third branch the calculations are almost similar, but  $l(k)$ has a value  1 twice, doesn't take the value 3 and when $l(k)=2$, $l(j)$ is also 2, not 1.
$$\det Q_2 =\frac{4}{\prod\limits_i (n-n(i))^2 }\cdot \Big[ 4^{n/2-2}\frac {(n-2)!}{\big(\frac{n-2}{2}\big)!\big(n - \frac{n-2}{2}-1\big)!}\Big]\cdot \Big[4\Big(\frac{n}{2}-1\Big)\Big]^2\Big[4\Big(\frac{n}{4}-1\Big)\Big]\cdot$$ $$\cdot\prod_{t=4}^{\frac{n}{2}-1}\Big[4\Big(\frac{n}{t+1} - 1\Big)  \Big]=$$
$$=\frac {4^{n-2}}{(n-1)^{2n}\Big[ \frac {(n-2)!}{(\frac {n}{2})!}  \Big]^2\Big[ \frac {(n-4)!}{(\frac {n}{2}-1)!}(n-2)^2  \Big]^2}\cdot \frac {(n-2)!}{\big(\frac{n-2}{2}\big)!\big(\frac{n}{2}\big)!}\Big(\frac{n}{2}-1\Big)^2\Big(\frac{n}{4}-1\Big)\frac{\Big[ \frac {(n-5)!}{(\frac {n}{2}-1)!}  \Big]}{\Big[ \frac {(\frac {n}{2})!}{4!}  \Big]}=$$
$$=\frac {4^{n-3}6}{(n-1)^{2n}(n-2)^2(n-4)!(n-2)!}$$

Hereby
$$ \frac {\det Q_1}{\det Q_2}=\frac {1}{12} \frac {n^2}{(n-3)}$$
And we see that, asymptotically when  $n \to \infty$ first topology is more probable, because
$$\lim_{n \to \infty} \frac {\det Q_1}{\det Q_2} = \infty$$

For the more common case of the topology with 3 mustache, where "middle" mustache are located between  $k$-th and $(k+1)$-th edges in the path, connecting boundary vertices of the rest two mustaches, $2\le k\le n-4$, then 
$$\det Q_k = \frac{4^{n-2} (k+1)(n-k-1)}{2(n-1)^{2n} \bigl((n-2)!\bigr) (n-2)}$$ 
and
$$\frac {\det Q_{k_1}}{\det Q_{k_2}} = \frac {(k_1+1)(n-k_1-1)}{(k_2 + 1)(n-k_2-1)}$$
For fixed $k_1$ пїЅ $k_2$ the limit is finite. Moreover it is clear visible, that if the third mustache is closer to the edge, then the topology is less probable. If, for example, $k_1$ has an order of $n/2$, and $k_2 = const$, then the limit is infinite and this has been demonstrated above.

\section {Conclusion}

A probability measure on the set of topologies of minimal fillings for additive spaces has been introduced. The analytical formulas, permitting the calculation of probabilities of arbitrary topologies, are derived, making possible their comparison. As an example, the ratio of probabilities for two families of topologies with the same number of edges have been calculated as well as the asymptotic of the ratio.
It is clear, that if the number of vertices is more than 6, then there exist more than two topologies and, thus, to calculate final values of probabilities, the algorithm must be used for all topologies and values must be normalized, but these is not necessery for the practical cases. It is important to note, that results, obtained with analytical solution, coincide with the computer simulation for the computed cases.


\begin{thebibliography}{99}

\bibitem{Gromov}
\newblock Gromov M.,
\newblock Filling Riemannian manifolds,
\newblock J. Dif. Geom. 1983. 18. 1–147.

\bibitem{Basics}
\newblock А.О.Иванов, А.А.Тужилин,
\newblock Теория экстремальных сетей,
\newblock Москва---Ижевск: Институт компьютерных исследований, 2003.
\newblock ISBN 5-93972-292-X.

\bibitem{Sledstv}
\newblock А.О.Иванов, А.А.Тужилин, 
\newblock Одномерная проблема Громова о минимальном заполнении, 
\newblock матем. сборник, т.203, N 5, 65--118 (2012). 


\bibitem{Zahar}
\newblock З.Н.Овсянников, 
 \newblock Обобщенно аддитивные пространства, 
 \newblock готовится к печати. 
    
\bibitem{MFform}
\newblock А.Ю. Еремин, 
\newblock  Формула веса минимального заполнения конечного метрического
пространства, 
\newblock Матем. сборник, 2012. 
   
\bibitem{Unique}
\newblock К. А. Зарецкий, 
\newblock Построение дерева по набору расстояний между висячими вер-
шинами,  
\newblock УМН, 20:6 (1965), 90–92.    

 
\bibitem{NP}
\newblock Richard M. Karp,
\newblock Reducibility among combinatorial problems,
\newblock Complexity of Computer Computations: Proc. of a Symp. on the Complexity of Computer Computations, R. E. Miller and J. W. Thatcher, Eds., The IBM Research Symposia Series, New York, NY: Plenum Press, 1972, pp. 85-103. 

\bibitem{DESMT}
\newblock M. R. Garey, R. L. Graham and D. S. Johnson,
\newblock The Complexity of Computing Steiner Minimal Trees,
\newblock SIAM J. Appl. Math., 32 (1977), pp. 835-859.


\end{thebibliography}
 \end{document}